\documentclass[letterpaper, 10 pt, conference]{ieeeconf}  

\IEEEoverridecommandlockouts                              

\usepackage{amsmath, amssymb, epsfig, graphicx, bm}
\usepackage{cases}
\usepackage{float, subfigure}
\usepackage{color}
\usepackage{amsfonts}
\usepackage{cite, url}
\usepackage[all,cmtip]{xy}
\usepackage{color,hyperref}
\definecolor{darkblue}{rgb}{0,0,0.8}
\hypersetup{colorlinks,breaklinks,
	linkcolor=darkblue,urlcolor=darkblue,anchorcolor=darkblue,citecolor=darkblue}

\newcommand{\R}{\mathbb{R}}
\newcommand{\Gra}{\mathcal{G}}

\newcommand{\bx}{\mathbf{x}}
\newcommand{\by}{\mathbf{y}}

\newcommand{\barf}{f}

\newcommand{\df}{\nabla F}

\newcommand{\T}{\intercal}

\newcommand{\spa}[1]{\mathrm{span}\{#1\}}
\newcommand{\nul}[1]{\mathrm{null}\{#1\}}

\newtheorem{theorem}{Theorem}
\newtheorem{lemma}{Lemma}
\newtheorem{assumption}{Assumption}
\newtheorem{definition}{Definition}
\newtheorem{remark}{Remark}

\title{\LARGE \bf
	A Push-Pull Gradient Method for Distributed Optimization in Networks
}

\author{Shi Pu, Wei Shi, Jinming Xu, and Angelia Nedi{\'c}
	\thanks{*This work was supported in parts by the NSF grant CCF-1717391
		and by the ONR grant no.\ N00014-12-1-0998.}
	\thanks{Shi Pu, Wei Shi, Jinming Xu and Angelia Nedi{\'c} are with the School of Electrical, Computer, and Energy Engineering, Arizona State University,
		Tempe, AZ 85287, USA.
		{\tt\small (emails: shipu3@asu.edu, wshi36@asu.edu, jinming.xu@asu.edu, Angelia.Nedich@asu.edu)}}%
}

\newcommand{\mx}{\mathbf{x}}
\newcommand{\my}{\mathbf{y}}
\newcommand{\ox}{\bar{x}}
\newcommand{\oy}{\bar{y}}

\begin{document}

	\maketitle
	\thispagestyle{empty}
	\pagestyle{empty}

	\begin{abstract}
		In this paper, we focus on solving a distributed convex optimization problem in a network, where each agent has its own convex cost function and the goal is to minimize the sum of the agents' cost functions while obeying the network connectivity structure. In order to minimize the sum of the cost functions,
		we consider a new distributed gradient-based method where each node maintains two estimates, namely, an estimate of the optimal decision variable and an estimate of the gradient
		for the average of the agents' objective functions. 
		From the viewpoint of an agent, the information about the 
		decision variable is pushed to the neighbors, while the information about the gradients is pulled from 
		the neighbors (hence giving the name ``push-pull gradient method"). 
		The method unifies the algorithms with different types of distributed architecture, including decentralized (peer-to-peer), centralized (master-slave), and semi-centralized (leader-follower) architecture. We show that the algorithm converges linearly for strongly convex and smooth objective functions over a directed static network. 
		In our numerical test, the algorithm performs well even for time-varying directed networks.
	\end{abstract}
	{\color{red}This is a preliminary version of the paper \cite{pu2018push2}.}
	
	\section{Introduction}
	In this paper, we consider a system involving $n$ agents whose goal
	is to collaboratively solve the following problem:
	\begin{equation} \label{problem}
	\begin{array}{c}
	\min\limits_{x\in\R^p}~\barf(x):=\sum\limits_{i=1}^n f_i(x),
	\end{array}
	\end{equation}
	where $x$ is the global decision variable and each function $f_i: \R^p\rightarrow \R$ is convex and known by agent $i$ only.
	The agents are embedded in a communication network, and their
	goal is to obtain an optimal and consensual solution through local neighbor communications and information exchange. This local exchange is desirable in situations where privacy needs to be preserved, 
	or the exchange of a large amount of data is prohibitively expensive due to limited communication resources.
	
	To solve problem~\eqref{problem} in a networked system of $n$ agents, 
	many algorithms have been proposed under various assumptions on the objective functions 
	and the underlying network~\cite{boyd2011distributed,peng2016arock,Shi2014,Shi2015,Mokhtari2016,varagnolo2016newton,olshevsky2017linear,scaman2017optimal,uribe2017optimal,Xu2015,nedic2017geometrically,di2016next,Shi2015_2,li2017decentralized,nedic2015distributed,xi2015linear,Zeng2015,xu2017convergence,wu2016decentralized,pu2017flocking,pu2018distributed,nedic2017achieving,qu2017harnessing,xi2018linear,wei2013distributed,mokhtari2017network}. Centralized algorithms are discussed in~\cite{boyd2011distributed}, where extensive applications in learning can be found. Parallel, coordinated, and asynchronous algorithms are discussed in~\cite{peng2016arock} and the references therein.
	
	Our emphasis in the literature review is on the decentralized optimization 
	since our approach builds on a new understanding of the decentralized consensus-based methods 
	for directed communication networks. Most references, including \cite{Shi2014,Shi2015,Mokhtari2016,varagnolo2016newton,olshevsky2017linear,scaman2017optimal,uribe2017optimal,Xu2015,nedic2017geometrically,di2016next,Shi2015_2,li2017decentralized}, 
	often restrict the underlying network connectivity structure, 
	or more commonly require doubly stochastic mixing matrices.
	The work in~\cite{Shi2014} has been the first to demonstrate the linear convergence of an ADMM-based 
	decentralized optimization scheme. Reference~\cite{Shi2015} uses a gradient difference structure in the algorithm to provide the first-order decentralized optimization algorithm which is capable of achieving the typical convergence rates of a centralized gradient method, while 
	references~\cite{Mokhtari2016,varagnolo2016newton} deal with the second-order decentralized methods. 
	By using Nesterov's acceleration, reference~\cite{olshevsky2017linear} has obtained a method whose convergence time scales linearly in the number of agents $n$, which is the best scaling with $n$ currently known.
	More recently, for a class of so-termed dual friendly functions, 
	papers~\cite{scaman2017optimal,uribe2017optimal} have obtained an optimal decentralized consensus optimization algorithm whose dependency on the condition number\footnote{The condition number of a smooth and strongly convex function is the ratio of its gradient Lipschitz constant and its strong convexity constant.} $\kappa$ of the system's objective function $\sum_{i=1}^n f(x_i)$ achieves the best known scaling in the order of $O(\sqrt{\kappa})$.
	Work in~\cite{Shi2015_2,li2017decentralized} investigates proximal-gradient methods which can tackle \eqref{problem} with proximal friendly component functions. Paper~\cite{wu2016decentralized} extends
	the work in~\cite{Shi2014} to handle asynchrony and delays. 
	References~\cite{pu2017flocking,pu2018distributed} considers a stochastic variant of problem~\eqref{problem} 
	in asynchronous networks. A tracking technique has been recently employed to develop
	decentralized algorithms for tracking the average of the Hessian/gradient in second-order methods~\cite{varagnolo2016newton}, allowing uncoordinated step-size~\cite{Xu2015,nedic2017geometrically}, handling non-convexity \cite{di2016next}, and achieving linear convergence over 
	time-varying graphs~\cite{nedic2017achieving}.
	
	For directed graphs, to eliminate the need of constructing a doubly stochastic matrix in reaching consensus\footnote{Constructing a doubly stochastic matrix over a directed graph needs weight balancing which requires an independent iterative procedure across the network; consensus is a basic element in decentralized optimization.}, reference \cite{Kempe2003} proposes the push-sum protocol. Reference \cite{tsianos2012push} has been the first to propose a push-sum based distributed optimization algorithm for directed graphs. Then, based on the push-sum technique again, a decentralized subgradient method for time-varying directed graphs has been proposed and analyzed in~\cite{nedic2015distributed}. Aiming to improve convergence for a smooth objective function and a fixed directed graph, work in~\cite{xi2015linear,Zeng2015} modifies the algorithm from~\cite{Shi2015} with the push-sum technique, thus providing a new algorithm which converges linearly for a strongly convex objective function on a static graph. However, the algorithm has some stability issues, which have been resolved in~\cite{nedic2017achieving} in a more general setting of time-varying directed graphs.
	
	In this paper, we introduce a modified gradient-tracking algorithm for decentralized (consensus-based) optimization in directed graphs. Unlike the push-sum protocol, our algorithm uses a row stochastic matrix for the mixing of the decision variables, while it employs a column stochastic matrix for tracking the average gradients. 
	Although motivated by a fully decentralized scheme, we will show that our algorithm can work both in fully decentralized networks and in two-tier networks. The contributions of this paper include the design of a new decentralized algorithm\footnote{While completing the paper, we became aware of a recent work discussing an algorithm that is similar to ours \cite{xin2018linear}. We have independently arrived to our method and results.} and the establishment of its linear convergence for a static directed graph.
	We numerically evaluate our proposed algorithm for both static and time-varying graphs, and 
	find that the algorithm is competitive as compared to the linearly convergent algorithm developed in~\cite{nedic2017achieving}.
	
	The structure of this paper is as follows. We first provide notation and state basic assumptions in 
	Subsection~\ref{sec: problem}. Then, we introduce our algorithm in Section~\ref{sec: RC} along with the intuition of its design and some examples explaining how it relates to (semi-)centralized and decentralized optimization. 
	We establish the linear convergence of our algorithm in Section~\ref{sec: conv_analysis}, while
	in Section~\ref{sec: simulation} we conduct numerical test to verify our theoretical claims. 
	Concluding remarks are given in Section \ref{sec: conclusion}.
	
	
	\subsection{Notation and Assumptions}
	\label{sec: problem}
	Throughout the paper, vectors default to columns if not otherwise specified.
	Let each agent $i\in\{1,2,\ldots,n\}$ hold a local copy $x_i\in\mathbb{R}^p$ of the decision variable and an auxiliary variable $y_i\in\mathbb{R}^p$ tracking the average gradients, where their
	values at iteration $k$ are denoted by $x_{i,k}$ and $y_{i,k}$, respectively. Let
	\begin{eqnarray*}
		&\mx:=[x_1, x_2, \ldots,x_n]^{\T}\in\mathbb{R}^{n\times p},\\
		&\my:=[y_1, y_2, \ldots,y_n]^{\T}\in\mathbb{R}^{n\times p}.
	\end{eqnarray*}
	Define $F(\mx)$ to be an aggregate objective function of the local variables,
	i.e., $F(\mx):=\sum_{i=1}^nf_i(x_i)$, and write
	\begin{equation*}
	\nabla F(\mx):=\left[\nabla f_1(x_1), \nabla f_2(x_2), \ldots, \nabla f_n(x_n)\right]^{\T}\in\mathbb{R}^{n\times p}.
	\end{equation*}
	
	\begin{definition}\label{def: norm n p}
		Given an arbitrary vector norm $\|\cdot\|$ on $\mathbb{R}^n$, for any $\mx\in\mathbb{R}^{n\times p}$, we define
		\begin{equation*}
		\|\mx\|:=\left\|\left[\|\mx^{(1)}\|,\|\mx^{(2)}\|,\ldots,\|\mx^{(p)}\|\right]\right\|_2,
		\end{equation*}
		where $\mx^{(1)},\mx^{(2)},\ldots,\mx^{(p)}\in\mathbb{R}^n$ are columns of $\mx$, and $\|\cdot\|_2$ represents the $2$-norm.
	\end{definition}
	
	A directed graph is a pair $\mathcal{G}=(\mathcal{V},\mathcal{E})$, where $\mathcal{V}$ is the set of vertices (nodes) and the edge set $\mathcal{E}\subseteq \mathcal{V}\times \mathcal{V}$ consists of ordered pairs of vertices.	
	A directed tree is a directed graph where every vertex, except for the root, has only one parent. 
	A spanning tree of a directed graph is a directed tree that connects the root to all other vertices in the graph (see~\cite{godsil2013algebraic}).
	
	Given a nonnegative matrix $M\in\mathbb{R}^{n\times n}$, the directed 
	graph induced by the matrix $M$ is 
	denoted by $\mathcal{G}_M=(\mathcal{V}_M,\mathcal{E}_M)$, where 
	$\mathcal{V}_M=\{1,2,\ldots,n\}$ and $(j,i)\in\mathcal{E}_M$ iff $M_{ij}>0$. We let
	$\mathcal{R}_M$ denote the set of roots of all directed spanning trees in the graph $\mathcal{G}_M$.
	
	We use the following assumption on the functions $f_i$ in~\eqref{problem}.
	\begin{assumption}
		\label{asp; strconvex Lipschitz}
		Each $f_i$ is $\mu$-strongly convex and its gradient is $L$-Lipschitz continuous, i.e., for any $x,x'\in\mathbb{R}^p$,
		\begin{equation}
		\begin{split}
		& \langle \nabla f_i(x)-\nabla f_i(x'),x-x'\rangle\ge \mu\|x-x'\|^2,\\
		& \|\nabla f_i(x)-\nabla f_i(x')\|\le L \|x-x'\|.
		\end{split}
		\end{equation}
	\end{assumption}
	Under Assumption \ref{asp; strconvex Lipschitz}, there exists a unique optimal 
	solution $x^*\in\mathbb{R}^{1\times p}$ to problem~\eqref{problem}.	
	
	\section{A Push-Pull Gradient Method}
	\label{sec: RC}
	The aggregated form of the proposed algorithm, termed push-pull gradient method (Push-Pull), works as follows:
	Initialize with any $\mx_0$ and $\my_0=\df(\mx_0)$, and update according to the following rule for $k\ge0$,
	\begin{subequations}\label{algorithm RC}
		\begin{align}
		&\mx_{k+1} =  R(\mx_k-\alpha \my_k),\label{eq:x-update}\\
		&\my_{k+1} =  C\left(\my_k+\nabla F(\mx_{k+1})-\nabla F(\mx_k)\right),\label{eq:y-update}
		\end{align}
	\end{subequations}
	where $R,C\in\mathbb{R}^{n\times n}$.
	We make the following standard assumption on the matrices $R$ and $C$.
	\begin{assumption}
		\label{asp: stochastic}
		We assume that $R\in\mathbb{R}^{n\times n}$ is nonnegative\footnote{A matrix is nonnegative iff all its elements are nonnegative.} row-stochastic and $C\in\mathbb{R}^{n\times n}$ is nonnegative column-stochastic, i.e., $R\mathbf{1}=\mathbf{1}$ and $\mathbf{1}^{\T} C=\mathbf{1}^{\T}$.
	\end{assumption}
	\begin{lemma}
		\label{lem: eigenvectors u v}
		Under Assumption~\ref{asp: stochastic}, the matrix $R$ has a nonnegative left eigenvector $u^{\T}$ (w.r.t.\ eigenvalue $1$) with $u^{\T}\mathbf{1}=n$, and the matrix $C$ has a nonnegative right eigenvector $v$ (w.r.t. eigenvalue $1$) with $\mathbf{1}^{\T}v=n$ (see \cite{horn1990matrix}).
	\end{lemma}
	The next condition ensures that $1$ is a simple eigenvalue of both $R$ and $C$.
	\begin{assumption}
		\label{asp: positive diagonals}
		The diagonal entries of R and C are positive, i.e., $R_{ii}>0$ and  $C_{ii}>0$ for all $i\in \mathcal{V}$.
	\end{assumption}
	Finally, we give the condition on the structures of $\mathcal{G}_R$ and $\mathcal{G}_{C}$. This assumption is weaker than requiring that both $\mathcal{G}_R$ and $\mathcal{G}_{C}$ are strongly connected.
	\begin{assumption}
		\label{asp: nonempty root set}
		The graphs $\mathcal{G}_R$ and $\mathcal{G}_{C^T}$ each contain at least one spanning tree. Moreover, $\mathcal{R}_R\cap\mathcal{R}_{C^{\T}}\neq \emptyset$.
	\end{assumption}
	
	
	Supposing that we have a strongly connected communication graph $\mathcal{G}$, there are multiple ways to construct $\mathcal{G}_R$ and $\mathcal{G}_{C}$ satisfying Assumption \ref{asp: nonempty root set}. One trivial approach is to set $\mathcal{G}_R=\mathcal{G}_{C}=\mathcal{G}$. Another way is to pick at random $i_r\in\mathcal{V}$ and let $\mathcal{G}_R$ (respectively, $\mathcal{G}_{C}$) be a spanning tree (respectively, reversed spanning tree) contained in $\mathcal{G}$ with $i_r$ as its root.
	Once graphs $\mathcal{G}_R$ and $\mathcal{G}_{C}$ are established, matrices $R$ and $C$ can be designed accordingly.
	
	To provide some intuition for the development of this algorithm, let us consider the optimality condition for \eqref{problem} in the following form:
	\begin{subequations}\label{eq:opt_cond_consensus}
		\begin{align}
		&\mx^*\in\nul{I-R},\label{eq:oc_line1_consensus}\\
		&\mathbf{1}^{\T}\df(\mx^*)=\mathbf{0},\label{eq:oc_line2_consensus}
		\end{align}
	\end{subequations}
	where $R$ satisfies Assumption \ref{asp: stochastic}. Consider now the algorithm in~\eqref{algorithm RC}. 
	Suppose that the algorithm produces two sequences $\{\mx_k\}$ and $\{\my_k\}$ converging to some points $\mx_\infty$ and $\my_\infty$, respectively. Then from (\ref{eq:x-update}) and (\ref{eq:y-update}) we would have 
	\begin{subequations}\label{eq:converge_consensus}
		\begin{align}
		&(I-R)(\mx_\infty-\alpha\my_\infty)+\alpha \my_\infty=0, \label{eq:converge_line1_consensus}\\
		&(I-C)\my_\infty=0. \label{eq:converge_line2_consensus}
		\end{align}
	\end{subequations}
	If $\spa{I-R}$ and $\nul{I-\mathbf{C}}$ are disjoint\footnote{This is indeed a consequence of Assumption \ref{asp: nonempty root set}.}, from (\ref{eq:converge_consensus}) we would have $\mx_\infty\in\nul{I-R}$ and $\my_\infty=\mathbf{0}$. Hence $\mx_\infty$ satisfies the optimality condition in~\eqref{eq:oc_line1_consensus}. 
	Then by induction we know $\mathbf{1}^{\T}\df(\mx_\infty)=\mathbf{1}^{\T}\my_\infty=\mathbf{0}$, which is exactly the optimality condition in~\eqref{eq:oc_line2_consensus}.
	
	The structure of the algorithm in~\eqref{algorithm RC} is similar to that of the DIGing algorithm proposed 
	in~\cite{nedic2017achieving} with the mixing matrices distorted (doubly stochastic matrices split into 
	a row-stochastic matrix and a column-stochastic matrix). The $\mx$-update can be seen as 
	an inexact gradient step with consensus, while the $\my$-update can be viewed as a gradient tracking step. 
	Such an asymmetric $R$-$C$ structure design has already been used in the literature of average 
	consensus~\cite{cai2012average}. However, we can not analyze the proposed optimization algorithm 
	using linear dynamical systems since we have a nonlinear dynamics due to the gradient terms.
	
	We now show how the proposed algorithm (\ref{algorithm RC}) unifies different types of distributed architecture.
	For the fully decentralized case, suppose we have a graph $\mathcal{G}$ that is undirected and connected. Then $R$ and $C$ can be chosen as symmetric matrices, in which case the proposed algorithm degrades to the one considered in~\cite{nedic2017achieving}; if the graph is directed and strongly connected, we can set $\mathcal{G}_R=\mathcal{G}_{C}=\mathcal{G}$ and design the weights for $R$ and $C$ correspondingly.
	
	To illustrate the less straightforward situation of (semi)-centralized networks, let us give a simple example. 
	Consider a four-node star network composed by $\{1,2,3,4\}$ where node $1$ is situated at the center and nodes $2$, $3$, and $4$ are (bidirectionally) connected with node $1$ but not connected to each other. In this case, the matrix $R$ in our algorithm can be chosen as 
	\[R=\left[
	\begin{array}{cccc}
	1 & 0 &0&0\cr
	0.5& 0.5 & 0 & 0\cr
	0.5&0&0.5&0\cr
	0.5&0&0&0.5
	\end{array}\right]
	\]
	and 
	\[C=\left[
	\begin{array}{cccc}
	1&0.5&0.5&0.5\cr
	0&0.5&0&0\cr
	0&0&0.5&0\cr
	0&0&0&0.5
	\end{array}\right].
	\]
	For a graphical illustration, the corresponding network topologies of $\Gra_R$ and $\Gra_C$ are shown in Fig. \ref{fig:Toy}.
	\begin{figure}[h]
		\begin{center}
			\includegraphics[height=8em,clip = true, trim = 0in 0in 0in 0in]{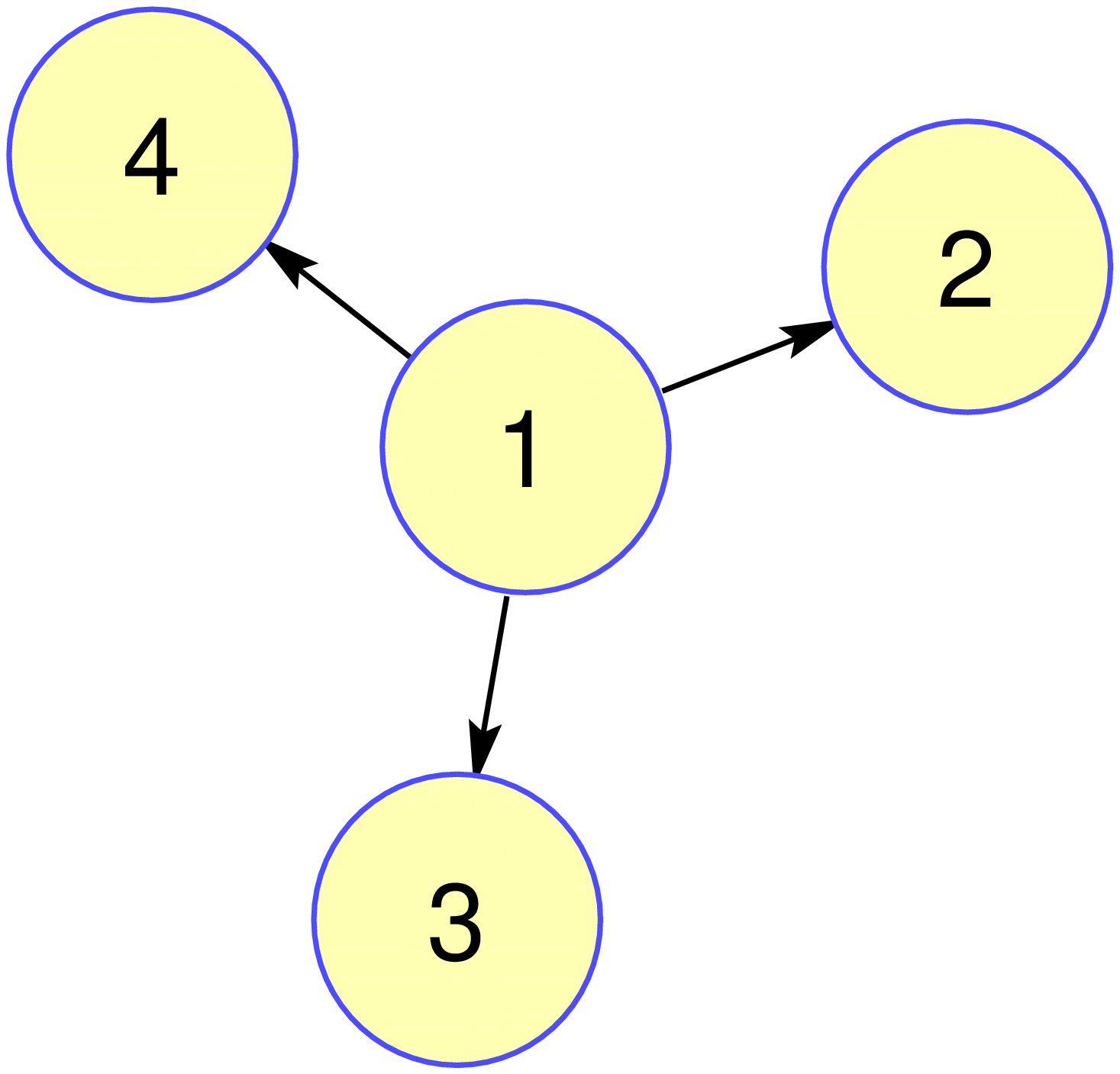}
			\includegraphics[height=8em,clip = true, trim = 0in 0in 0in 0in]{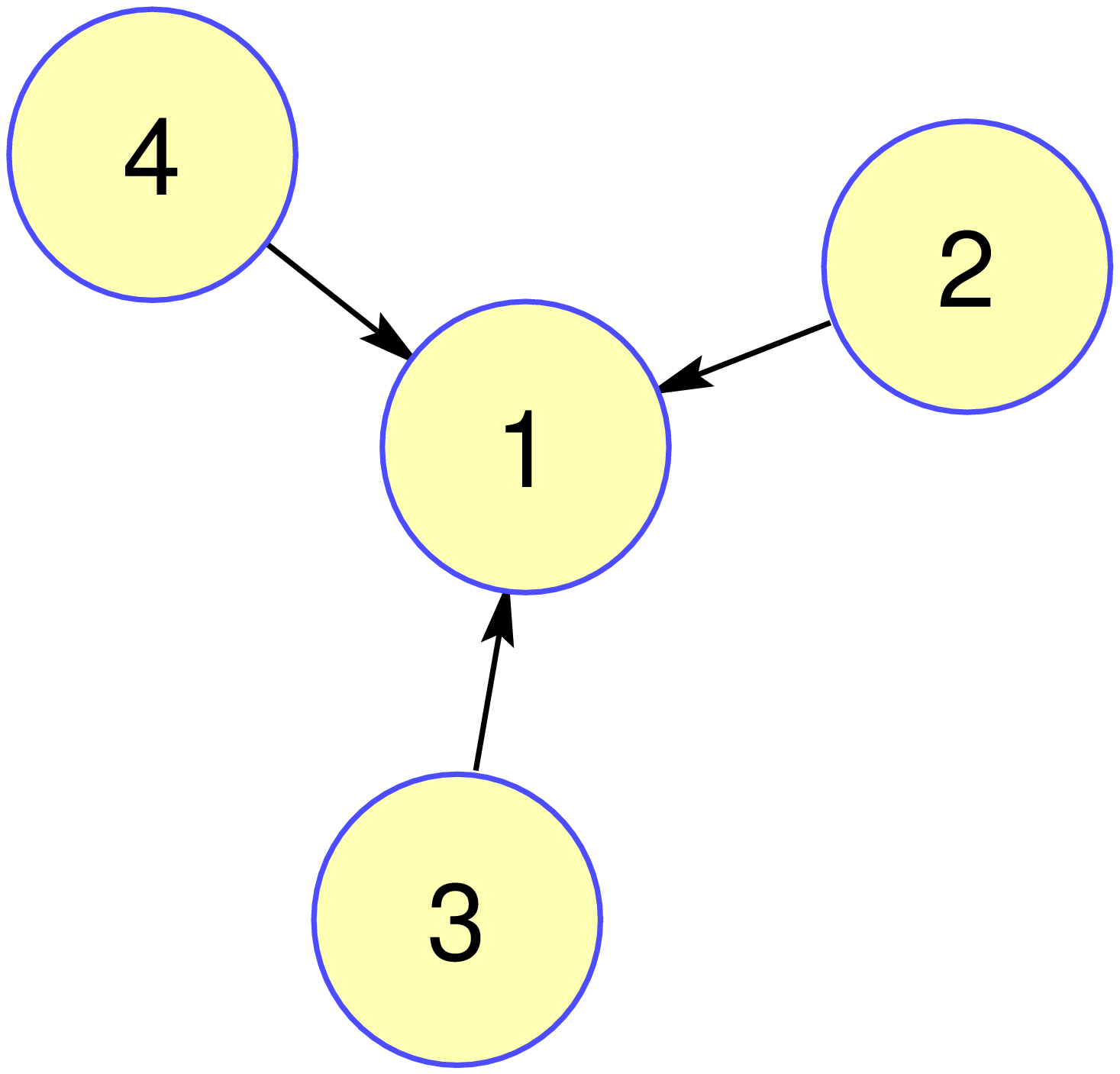}
			\caption{The left is $\Gra_R$ and the right is $\Gra_C$.}\label{fig:Toy}
		\end{center}
	\end{figure}
	The central node $1$ pushes (diffuses) information regarding $x_{1,k}$ to the neighbors (in this case the entire network) through $\Gra_R$, while the others can only passively infuse the information from node $1$.
	At the same time, node $1$ pulls (collects) information regarding $y_{i,k}$ ($i=2,3,4$) from the neighbors
	through $\Gra_C$, while the other nodes can only actively comply with the request from node $1$. 
	This motivates the algorithm's name {\it push-pull gradient method}.
	Although nodes $2$, $3$, and $4$ are updating their $y_i$'s accordingly, these quantities do not have to contribute to the optimization procedure and will die out geometrically fast due to the weights in the last three rows of $C$. Consequently, in this special case, the local step-size $\alpha$ for agents $2$, $3$, and $4$ can be set to $0$. Without loss of generality, suppose $f_1(x)=0, \forall x$. Then the algorithm becomes a typical centralized algorithm for minimizing $\sum_{i=2}^4 f_i(x)$ where the master node $1$ utilizes the slave nodes $2$, $3$, and $4$ to compute the gradient information in a distributed way.
	
	
	Taking the above as an example for explaining the semi-centralized case, it is worth nothing that node $1$ can be replaced by a strongly connected subnet in $\Gra_R$ and $\Gra_C$, respectively. 
	Correspondingly, nodes $2$, $3$, and $4$ can all be replaced by subnets as long as the information from the master layer in these subnets can be diffused to all the slave layer agents in $\Gra_R$, while the information from all the slave layer agents can be diffused to the master layer in $\Gra_C$.
	Specific requirements on connectivities of slave subnets can be understood by using the concept of rooted trees. We refer to the nodes as leaders if their roles in the network are similar to the role of node $1$; and the other nodes are termed as followers. Note that after the replacement of the individual nodes by subnets, the network structure in all subnets are decentralized, while the relationship between leader subnet and follower subnets is master-slave. This is why we refer to such an architecture as semi-centralized.
	\begin{remark}
		There can be multiple variants of the proposed algorithm depending on whether the Adapt-then-Combine (ATC) strategy \cite{Sayed2013} is used in the $\bx$-update and/or the $\by$-update (see Remark 3 in~\cite{nedic2017achieving} for more details). Our following analysis can be easily adapted for these variants. We have also tested one of the variants in Section \ref{sec: simulation}.
	\end{remark}
	
	\section{Convergence Analysis}\label{sec: conv_analysis}
	In this section, we study the convergence properties of the proposed algorithm.
	We first define the following variables:
	\begin{eqnarray*}
		\ox_k  :=  \frac{1}{n}u^{\T} \mx_k,\ \
		\oy_k  :=  \frac{1}{n}\mathbf{1}^{\T}\my_k.
	\end{eqnarray*}
	Our strategy is to bound $\|\ox_{k+1}-x^*\|_2$, $\|\mx_{k+1}-\mathbf{1}\ox_{k+1}\|_R$ and $\|\my_{k+1}-v\oy_{k+1}\|_C$ in terms of linear combinations of their previous values, where $\|\cdot\|_R$ and $\|\cdot\|_C$ are specific norms to be defined later. In this way we establish a linear system of inequalities which allows us to derive the convergence results. The proof technique was inspired by \cite{qu2017harnessing,xi2018linear}.
	
	\subsection{Preliminary Analysis}
	
	From the algorithm (\ref{algorithm RC}) and Lemma \ref{lem: eigenvectors u v}, we have
	\begin{equation}
	\label{ox_k+1 pre}
	\ox_{k+1}=\frac{1}{n}u^{\T} R(\mx_k-\alpha \my_k)=\ox_k-\frac{\alpha}{n}u^{\T}\my_k,
	\end{equation}
	and
	\begin{equation}
	\begin{aligned}
	\label{oy_ave}
	\oy_{k+1}&=\frac{1}{n}\mathbf{1}^{\T} C\left(\my_k+\nabla F(\mx_{k+1})-\nabla F(\mx_k)\right)\\
	&=\oy_k+\frac{1}{n}\mathbf{1}^{\T}\left(\nabla F(\mx_{k+1})-\nabla F(\mx_k)\right).
	\end{aligned}
	\end{equation}
	With the initialization $\my_0=\nabla F(\mx_0)$, we obtain by induction
	\begin{equation}
	\label{oy_k}
	\oy_k=\frac{1}{n}\mathbf{1}^{\T}\nabla F(\mx_{k}), \ \ \forall k.
	\end{equation}
	Let us further define $g_k:=\frac{1}{n}\mathbf{1}^{\T}\nabla F(\mathbf{1}\ox_k)$. Then, we obtain from relation (\ref{ox_k+1 pre})
	\begin{multline}
	\begin{aligned}
	\label{ox pre}
	&\ox_{k+1}=\ox_k-\frac{\alpha}{n}u^{\T}\left(\my_k-v\oy_k+v\oy_k\right)\\
	&=\ox_k-\frac{\alpha}{n}u^{\T}v\oy_k-\frac{\alpha}{n}u^{\T}\left(\my_k-v\oy_k\right)\\
	&=\ox_k-\alpha'g_k-\alpha'(\oy_k-g_k)-\frac{\alpha}{n}(u-\mathbf{1})^{\T}\left(\my_k-v\oy_k\right),
	\end{aligned}
	\end{multline}
	where 
	\begin{equation}
	\label{alpha'}
	\alpha':=\frac{\alpha}{n} u^{\T}v.
	\end{equation}
	We will show later that Assumption \ref{asp: nonempty root set} ensures $\alpha'>0$.
	
	In view of \eqref{algorithm RC} and Lemma \ref{lem: eigenvectors u v}, using \eqref{ox_k+1 pre} we have
	\begin{multline}
	\label{mx-ox pre}
	\mx_{k+{}1}-\mathbf{1}\ox_{k+1}=R(\mx_k-\alpha \my_k)-\mathbf{1}\ox_k+\frac{\alpha}{n}\mathbf{1}u^{\T}\my_k\\
	=R(\mx_k-\mathbf{1}\ox_k)-\alpha\left(R-\frac{\mathbf{1}u^{\T}}{n}\right)\my_k\\
	=\left(R-\frac{\mathbf{1}u^{\T}}{n}\right)(\mx_k-\mathbf{1}\ox_k)-\alpha\left(R-\frac{\mathbf{1}u^{\T}}{n}\right)\left(\my_k-\mathbf{1}\oy_k\right),
	\end{multline}
	and from \eqref{oy_ave} we obtain
	\begin{multline}
	\label{my-oy pre}
	\my_{k+1}-v\oy_{k+1}=C\my_k-v\oy_k\\
	+\left(C-\frac{v\mathbf{1}^{\T}}{n}\right)\left(\nabla F(\mx_{k+1})-\nabla F(\mx_k)\right)\\
	=\left(C-\frac{v\mathbf{1}^{\T}}{n}\right)(\my_k-v\oy_k)\\
	+\left(C-\frac{v\mathbf{1}^{\T}}{n}\right)\left(\nabla F(\mx_{k+1})-\nabla F(\mx_k)\right).
	\end{multline}
	\subsection{Supporting Lemmas}
	Before proceeding to the main results, we state a few useful lemmas.
	\begin{lemma}
		\label{lem: Lipschitz implications}
		Under Assumption \ref{asp; strconvex Lipschitz}, there holds
		\begin{align}
		\|\oy_k-g_k\|_2 & \le \frac{L}{\sqrt{n}}\|\mx_k-\mathbf{1}\ox_k\|_2,\\
		\|g_k\|_2 & \le L\|\ox_k-x^*\|_2.
		\end{align}
		In addition, when $\alpha'\le 2/(\mu+L)$, we have
		\begin{equation}
		\|\ox_k-\alpha'g_k-x^*\|_2 \le (1-\alpha'\mu)\|\ox_k-x^*\|_2, \ \ \forall k.
		\end{equation}
	\end{lemma}
	\begin{proof}
		See Appendix \ref{subsec: proof lemma Lipschitz implications}.
	\end{proof}
	\begin{lemma}
		\label{lem: condition u v >0}
		Suppose Assumption \ref{asp: stochastic} holds, and assume that
		$\mathcal{R}_R \neq\emptyset$ and $\mathcal{R}_{C^{\T}}\neq\emptyset$. Then, $\mathcal{R}_R\cap\mathcal{R}_{C^{\T}}\neq \emptyset$ iff $u^{\T} v>0$.
	\end{lemma}
	\begin{proof}
		See Appendix \ref{subsec: proof condition u v >0}.
	\end{proof}
	Lemma \ref{lem: condition u v >0} explains why Assumption \ref{asp: nonempty root set} is essential for the Push-Pull algorithm (\ref{algorithm RC}) to work. Without the condition, $\alpha'=0$ by its definition in (\ref{alpha'}).
	\begin{lemma}
		\label{lem: spectral radii}
		Suppose Assumptions \ref{asp: stochastic}-\ref{asp: nonempty root set} hold. Let
		$\rho_R$ and $\rho_C$ be the spectral radii of $(R-\mathbf{1}u^{\T}/n)$ and $(C-v\mathbf{1}^{\T}/n)$, respectively. Then, we have $\rho_R<1$ and $\rho_C<1$.
	\end{lemma}
	\begin{proof}
		See Appendix \ref{subsec: proof lemma spectral radii}.
	\end{proof}
	\begin{lemma}
		There exist matrix norms $\|\cdot\|_R$ and $\|\cdot\|_C$ 
		such that $\sigma_R:=\|R-\frac{\mathbf{1}u^{\T}}{n}\|_R<1$, $\sigma_C:=\|C-\frac{v\mathbf{1}^{\T}}{n}\|_C<1$, and $\sigma_R$ and $\sigma_C$ are arbitrarily close to $\rho_R$ and $\rho_C$, respectively.
	\end{lemma}
	\begin{proof}
		See \cite[Lemma 5.6.10]{horn1990matrix}  and the discussions thereafter.
	\end{proof}
	In the rest of this paper, with a slight abuse of notation,
	we do not distinguish between the vector norms on $\mathbb{R}^n$ and their induced
	matrix norms.
	\begin{lemma}
		\label{lem: matrix norm production}
		Given an arbitrary norm $\|\cdot\|$, for any $W\in\mathbb{R}^{n\times n}$ and $\mx\in\mathbb{R}^{n\times p}$, we have $\|W\mx\|\le \|W\|\|\mx\|$. For any $w\in\mathbb{R}^{n\times 1}$ and $x\in\mathbb{R}^{1\times p}$, we have $\|wx\|=\|w\|\|x\|_2$.
	\end{lemma}
	\begin{proof}
		See Appendix \ref{subsec: proof lemma matrix norm production}.
	\end{proof}
	\begin{lemma}
		\label{lem: norm equivalence}
		There exist constants $\delta_{C,R}, \delta_{C,2}, \delta_{R,C}, \delta_{R,2}>0$ such that for all $\mx\in\mathbb{R}^{n\times p}$, we have  $\|\mx\|_C\le \delta_{C,R}\|\mx\|_R$, $\|\mx\|_C\le \delta_{C,2}\|\mx\|_2$, $\|\mx\|_R\le \delta_{R,C}\|\mx\|_C$,  and $\|\mx\|_R\le \delta_{R,2}\|\mx\|_2$. In addition, with 
		a proper rescaling of the norms $\|\cdot\|_R$ and $\|\cdot\|_C$, we have
		$\|\mx\|_2\le \|\mx\|_R$ and $\|\mx\|_2\le \|\mx\|_C$.
	\end{lemma}
	\begin{proof}
		The result follows from the equivalence relation of all norms on $\mathbb{R}^n$ and Definition \ref{def: norm n p}.
	\end{proof}
	\subsection{Main Results}
	The following lemma establishes a linear system of inequalities that bound $\|\ox_{k+1}-x^*\|_2$, $\|\mx_{k+1}-\mathbf{1}\ox_k\|_R$ and $\|\my_{k+1}-v\oy_k\|_C$.
	\begin{lemma}
		\label{lem: important inequalities}
		Under Assumptions \ref{asp; strconvex Lipschitz}-\ref{asp: nonempty root set}, when $\alpha'\le 2/(\mu+L)$, we have the following linear system of inequalities:
		\begin{equation}
		\label{main ineqalities}
		\begin{bmatrix}
		\|\ox_{k+1}-x^*\|_2\\
		\|\mx_{k+1}-\mathbf{1}\ox_{k+1}\|_R\\
		\|\my_{k+1}-v\oy_{k+1}\|_C
		\end{bmatrix}
		\le
		A
		\begin{bmatrix}
		\|\ox_k-x^*\|_2\\
		\|\mx_k-\mathbf{1}\ox_k\|_R\\
		\|\my_k-v\oy_k\|_C
		\end{bmatrix},
		\end{equation}
		where the inequality is to be taken component-wise, and elements of the transition matrix $A=[a_{ij}]$ are given by:
		\begin{eqnarray*}
			\begin{bmatrix}
				a_{11}\\
				a_{21}\\
				a_{31}
			\end{bmatrix} & = & 
			\begin{bmatrix}
				1-\alpha'\mu\\
				\alpha\sigma_R \|v-\mathbf{1}\|_R L\\
				\alpha\sigma_C \delta_{C,2}\|Rv\|_2 L^2
			\end{bmatrix},\\
			\begin{bmatrix}
				a_{12}\\
				a_{22}\\
				a_{32}
			\end{bmatrix} & = &
			\begin{bmatrix}
				\frac{\alpha'L}{\sqrt{n}}\\
				\sigma_R\left(1+\alpha\|v-\mathbf{1}\|_R \frac{L}{\sqrt{n}}\right)\\
				\sigma_C \delta_{C,2}L\left(\|R-I\|_2+\alpha\|Rv\|_2\frac{L}{\sqrt{n}}\right)
			\end{bmatrix},\\
			\begin{bmatrix}
				a_{13}\\
				a_{23}\\
				a_{33}
			\end{bmatrix} & = &
			\begin{bmatrix}
				\frac{\alpha\|u-\mathbf{1}\|_2}{n}\\
				\alpha\sigma_R\delta_{R,C}\\
				\sigma_C\left(1+\alpha\delta_{C,2}\|R\|_2 L\right)
			\end{bmatrix}.
		\end{eqnarray*}
	\end{lemma}
	\begin{proof}
		See Appendix \ref{proof: important inequalities}.
	\end{proof}
	
	In light of Lemma \ref{lem: important inequalities}, $\|\ox_k-x^*\|_2$, $\|\mx_k-\mathbf{1}\ox_k\|_R$ and $\|\my_k-v\oy_k\|_C$ all converge to $0$ linearly at rate $\mathcal{O}(\rho_A^k)$ if the spectral radius of $A$ satisfies $\rho_A<1$. 
	The next lemma provides some sufficient conditions for the relation $\rho_A<1$ to hold.
	\begin{lemma}
		\label{lem: rho_M}
		Given a nonnegative, irreducible matrix $M=[m_{ij}]\in\mathbb{R}^{3\times 3}$ with $m_{11},m_{22},m_{33}<\lambda^*$ for some $\lambda^*>0$. A necessary and sufficient condition for $\rho_M<\lambda^*$ is $\text{det}(\lambda^* I-M)>0$.
	\end{lemma}
	\begin{proof}
		See Appendix \ref{subsec: proof lemma rho_M}.
	\end{proof}
	
	Now, we are ready to deliver our main convergence result for the Push-Pull algorithm in~\eqref{algorithm RC}.
	
	\begin{theorem}\label{theory:main}
		Suppose Assumptions \ref{asp; strconvex Lipschitz}-\ref{asp: nonempty root set} hold and
		\begin{equation}
		\alpha \le \min\left\{\frac{2c_3}{c_2+\sqrt{c_2^2+4c_1c_3}}, \frac{(1-\sigma_C)}{2\sigma_C\delta_{C,2}\|R\|_2 L}\right\},
		\end{equation}
		where $c_1,c_2,c_3$ are given in (\ref{c1})-(\ref{c3}).
		Then, the quantities $\|\ox_k-x^*\|_2$, $\|\mx_k-\mathbf{1}\ox_k\|_R$ and $\|\my_k-v\oy_k\|_C$ all converge to $0$ at the linear rate $\mathcal{O}(\rho_A^k)$ 
		with $\rho_A<1$, where $\rho_A$ denotes the spectral radius of $A$. 
	\end{theorem}
	
	\begin{proof}
		In light of Lemma \ref{lem: rho_M}, it suffices to ensure $a_{11},a_{22},a_{33}<1$ and $\text{det}(I-A)>0$, or equivalently
		{\small\begin{multline}
			\label{|I-A|>0}
			\text{det}(I-A)=(1-a_{11})(1-a_{22})(1-a_{33})-a_{12}a_{23}a_{31}\\
			-a_{13}a_{21}a_{32}-(1-a_{22})a_{13}a_{31}-(1-a_{11})a_{23}a_{32}-(1-a_{33})a_{12}a_{21}\\
			=(1-a_{11})(1-a_{22})(1-a_{33})
			-\alpha'\alpha^2\sigma_R\sigma_C\delta_{R,C}\delta_{C,2}\|Rv\|_2\frac{L^3}{\sqrt{n}}\\
			-\alpha^2\sigma_R\sigma_C\delta_{C,2} \|u-\mathbf{1}\|_2\|v-\mathbf{1}\|_R\left(\|R-I\|_2+\alpha\|Rv\|_2\frac{L}{\sqrt{n}}\right)\frac{L^2}{n}\\
			-\alpha^2\sigma_C\delta_{C,2}\|Rv\|_2\|u-\mathbf{1}\|_2\frac{L^2}{n}(1-a_{22})\\
			-\alpha\sigma_R\sigma_C \delta_{R,C}\delta_{C,2}L\left(\|R-I\|_2+\alpha\|Rv\|_2\frac{L}{\sqrt{n}}\right)(1-a_{11})\\
			-\alpha'\alpha\sigma_R \|v-\mathbf{1}\|_R\frac{L^2}{\sqrt{n}}(1-a_{33})>0.
			\end{multline}}\normalsize
		We now provide some sufficient conditions under which $a_{11},a_{22},a_{33}<1$ and (\ref{|I-A|>0}) holds true.
		First, $a_{11}<1$ is ensured by choosing $\alpha'\le 2/(\mu+L)$. let
		\begin{eqnarray}
		1-a_{22} & \ge & \frac{1}{2}(1-\sigma_R),\\
		1-a_{33} & \ge & \frac{1}{2}(1-\sigma_C).
		\end{eqnarray}
		We get
		\begin{equation}
		\label{alpha loose condition}
		\alpha\le \min\left\{\frac{(1-\sigma_R)\sqrt{n}}{2\sigma_R\|v-\mathbf{1}\|_R L},\frac{(1-\sigma_C)}{2\sigma_C\delta_{C,2}\|R\|_2 L}\right\}.
		\end{equation}
		
		Second, notice that $a_{22}>\sigma_R$ and $a_{33}>\sigma_C$. A sufficient condition for $\text{det}(I-A)>0$ is to substitute the first $(1-a_{22})$ (respectively, $(1-a_{33})$) in (\ref{|I-A|>0}) by $(1-\sigma_R)/2$ (respectively, $(1-\sigma_C)/2$), and substitute the second $(1-a_{22})$ (respectively, $(1-a_{33})$) by $(1-\sigma_R)$ (respectively, $(1-\sigma_C)$). We have 
		\begin{equation*}
		c_1\alpha^2+c_2\alpha-c_3<0,
		\end{equation*}
		where
		\begin{multline}
		\label{c1}
		c_1 = \frac{u^{\T} v}{n}\sigma_R\sigma_C\delta_{R,C}\delta_{C,2}\|Rv\|_2\frac{L^3}{\sqrt{n}}\\
		+\sigma_R\sigma_C\delta_{C,2} \|u-\mathbf{1}\|_2\|v-\mathbf{1}\|_R\|Rv\|_2\frac{L}{\sqrt{n}}\frac{L^2}{n}\\
		+\frac{u^{\T} v}{n}\mu\sigma_R\sigma_C \delta_{R,C}\delta_{C,2}L\|Rv\|_2\frac{L}{\sqrt{n}}\\
		=\sigma_R\sigma_C\delta_{C,2}\|Rv\|_2\frac{L^2}{n\sqrt{n}}\left[u^{\T} v\delta_{R,C}(L+\mu)\right.\\
		\left.+\|u-\mathbf{1}\|_2\|v-\mathbf{1}\|_R L\right],
		\end{multline}
		\begin{multline}
		\label{c2}
		c_2=\sigma_R\sigma_C\delta_{C,2} \|u-\mathbf{1}\|_2\|v-\mathbf{1}\|_R\|R-I\|_2\frac{L^2}{n}\\
		+\sigma_C\delta_{C,2}\|Rv\|_2\|u-\mathbf{1}\|_2(1-\sigma_R)\frac{L^2}{n}\\
		+\sigma_R\sigma_C \delta_{R,C}\delta_{C,2}L\|R-I\|_2\frac{u^{\T} v}{n}\mu\\
		+\sigma_R \|v-\mathbf{1}\|_R\frac{L^2}{\sqrt{n}}(1-\sigma_C)\frac{u^{\T} v}{n},
		\end{multline}
		and 
		\begin{equation}
		\label{c3}
		c_3=\frac{u^{\T} v}{4n}\mu(1-\sigma_R)(1-\sigma_C).
		\end{equation}
		Hence
		\begin{equation}
		\label{alpha strict condition}
		\alpha\le \frac{2c_3}{c_2+\sqrt{c_2^2+4c_1c_3}}.
		\end{equation}
		Relations (\ref{alpha loose condition}) and (\ref{alpha strict condition}) yield the final bound on $\alpha$.
	\end{proof}
	\begin{remark}
		When $\alpha$ is sufficiently small, it can be shown that $\rho_A\simeq 1-\alpha'\mu$, in which case the Push-Pull algorithm is comparable to its centralized counterpart with step-size $\alpha'$.
	\end{remark}
	\section{SIMULATIONS}
	\label{sec: simulation}
	In this section, we provide numerical comparisons of a few different algorithms under various network settings. Our settings for objective functions are the same as that described in \cite{nedic2017achieving}. Each node in the network holds a Huber-typed objective function $f_i(x)$ and the goal is to optimize the total Huber loss $f(x)=\sum_{i=1}^n f_i(x)$. The objective functions $f_i$'s are randomly generated but are manipulated such that the global optimizer $x^*$ is located at the $\ell_2^2$ zone of $f(x)$ while the origin (which is set to be the initial state of $\bx_k$ for all involved algorithms) is located outside of that zone. 
	
	We first conduct an experiment over time-invariant directed graphs. The network is generated randomly with $12$ nodes and $24$ unidirectional links (at most $12\times11=131$ possible links in this case) and is guaranteed to be strongly connected. We test our proposed algorithm, Push-Pull, against Push-DIGing \cite{nedic2017achieving} and Xi-Row \cite{xi2018linear}. Among these algorithms, Push-DIGing is a push-sum based algorithm which only needs push operations for information dissemination in the network; Xi-row is an algorithm that only uses row stochastic mixing matrices and thus only needs pull operations to fetch information in the network; in comparison, our algorithm needs the network to support both push operations and pull operations. The per-node storage complexity of Push-Pull (or Push-DIGing) is $O(p)$ while that of Xi-row is $O(n+p)$. Note that at each iteration, the amount of data transmitted over each link also scales at such orders for these algorithms, respectively. For large-scale networks ($n\gg p$), Xi-row may suffer from high needs in storage/bandwidth and/or become under limited transmission rates. The evolution of the (normalized) residual $\frac{\|\bx_k-\bx^*\|_2^2}{\|\bx_0-\bx^*\|_2^2}$ is illustrated in Fig. \ref{eps:Conv_TI}. The step-sizes are hand-tuned for all the algorithms to optimize the convergence speed.
	\begin{figure}[h]
		\begin{center}
			\centering
			\includegraphics[width=3.2in]{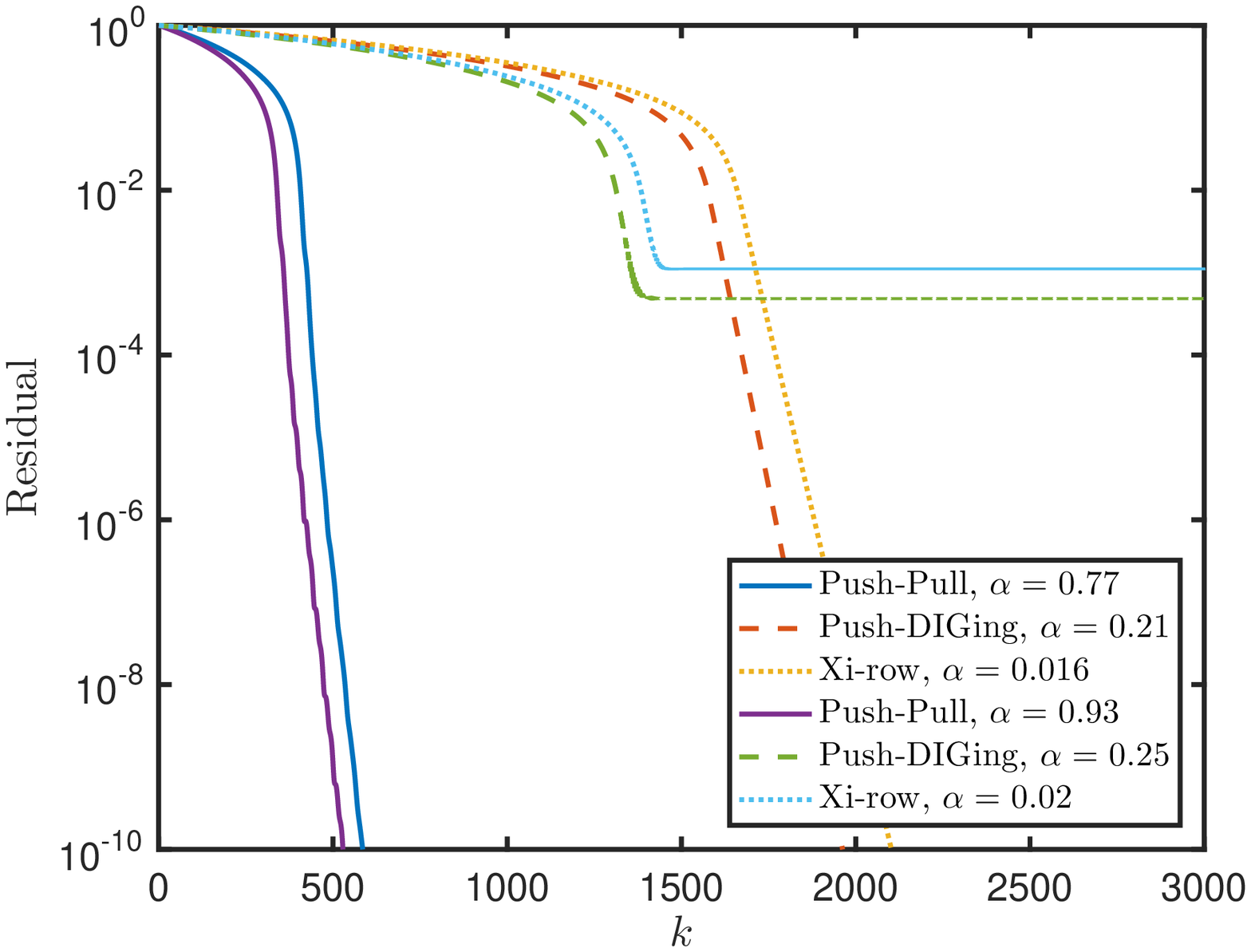}
			\caption{Plots of (normalized) residuals against number of iterations over a time-invariant directed graph.\label{eps:Conv_TI}}
		\end{center}
	\end{figure}
	
	Although our algorithm is designed and analyzed over time-invariant directed graphs, its extension to time-varying directed graphs is straightforward. Let us use the above generated directed graph as a base graph. To test our theory for a leader-follower architecture, we randomly select multiple nodes as leaders and randomly add enough links between the leaders (in this example, number of leaders is $2$) so that they form a strongly connected subnet. Then at each iteration, only $50\%$ randomly chosen links will be activated. In Fig. \ref{eps:Conv_TV}, we plot the performance of Push-Pull-half (a variant of Push-Pull where the ATC strategy is not employed in the $\by$-update; it needs only one round of communication at each iteration) and Push-DIGing without considering the leader-follower structure. That is, for Push-Pull-half and Push-DIGing, a time-varying directed graph sequence based on random link activation is used where the underlying graph is strongly connected. 
	\begin{figure}[h]
			\centering
			\includegraphics[width=3.2in]{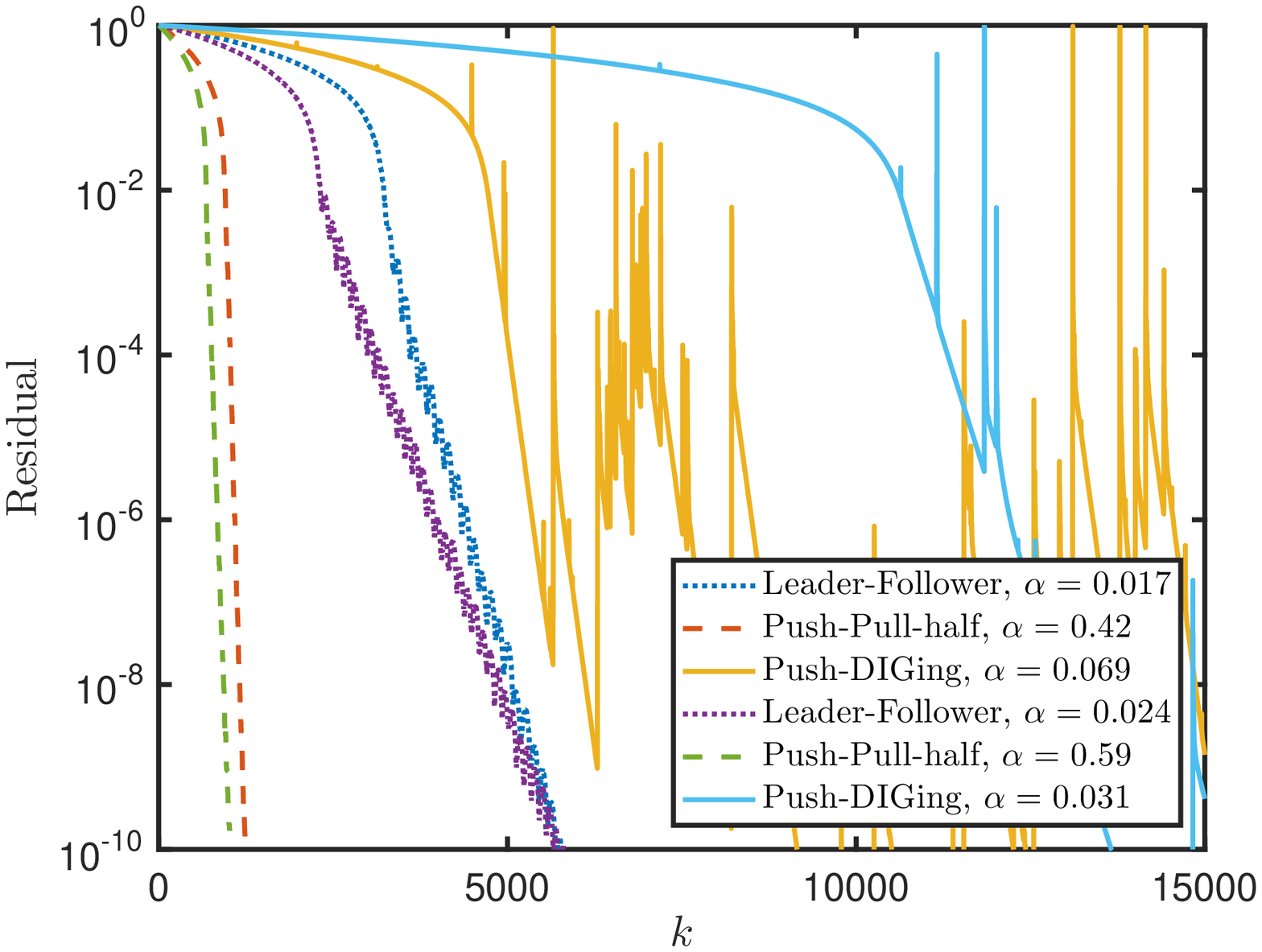}
			\caption{Plots of (normalized) residuals against number of iterations over a time-varying directed graph sequence. \label{eps:Conv_TV}}
	\end{figure}
	
	Then in Fig. \ref{eps:Conv_TV} we further show the performance of Push-Pull under the leader-follower architecture. The major difference on the graph sequence is that, in the leader-follower architecture, all the outbound information links of the leader subnet are not used when performing the $\by$-update; all the inbound information links of the leader subnet are not used when performing the $\bx$-update. Note that in such a way, the union of all directed graphs corresponding to $R_k$ (or $C_k$) is not strongly connected. The numerical results show, as expected, that the convergence of Push-Pull under the leader-follower architecture is slower than that of Push-Pull-half with strongly connected underlying graphs.
	
	In the experiment, we observe that there are many spikes on the residual curve of Push-DIGing. Push-DIGing for time-varying graphs can be numerically unstable due to the use of division operations in the algorithm and the divisors can scale badly at the order of $\Omega(n^{-Bn})$ where $n$ is the number of nodes and $B$ a bounded constant that describes the connectivity of time-varying graphs (the smaller $B$ is, the better the network is connected; see \cite{nedic2017achieving} for the definition of $B$). A simple network with number of nodes $n=15$ and time-varying constant $B=10$ will easily give a number that is recognized by common computers (using double-precision floating-point format) as $0$. As a contrast, in Push-Pull, there is no divisors that scale at such level. In addition, the theoretical upper bound on the step-size of Push-DIGing also scales at the order of $O(n^{-Bn})$. This implies that Push-DIGing will not work well for either large scale networks or time-varying networks with large variations.
	
	\section{Conclusions}
	\label{sec: conclusion}
	In this paper, we have studied the problem of distributed optimization over a network.
	In particular, we proposed a new distributed gradient-based method (Push-Pull) where each node maintains estimates of the optimal decision variable and the average gradient of the agents' objective functions. From the viewpoint of an agent, the information about the 
	decision variable is pushed to its neighbors, while the information about the gradients is pulled from its neighbors. This method works for different types of distributed architecture, including decentralized, centralized, and semi-centralized  architecture. We have showed that the algorithm converges linearly for strongly convex and smooth objective functions over a directed static network. In the simulations, we have demonstrated the effectiveness of the proposed algorithm for both static and time-varying directed networks.
	


	

	
	\bibliographystyle{IEEEtran}
	\bibliography{mybib2}
	
	\section{APPENDIX}
	
	\subsection{Proof of Lemma \ref{lem: Lipschitz implications}}
	\label{subsec: proof lemma Lipschitz implications}
	In light of Assumption \ref{asp; strconvex Lipschitz} and (\ref{oy_k}),
	\begin{multline*}
	\|\oy_k-g_k\|_2=\frac{1}{n}\|\mathbf{1}^{\T}\nabla F(\mx_{k})-\mathbf{1}^{\T}\nabla F(\mathbf{1}\ox_k)\|_2\\
	\le \frac{L}{n}\sum_{i=1}^n \|x_{i,k}-\ox_k\|_2\le \frac{L}{\sqrt{n}}\|\mx_k-\mathbf{1}\ox_k\|_2,
	\end{multline*}
	and
	\begin{multline*}
	\|g_k\|_2=\frac{1}{n}\|\mathbf{1}^{\T}\nabla F(\mathbf{1}\ox_k)-\mathbf{1}^{\T}\nabla F(\mathbf{1}x^*)\|_2\\
	\le \frac{L}{n}\sum_{i=1}^n \|\ox_k-x^*\|_2=L\|\ox_k-x^*\|_2.
	\end{multline*}
	Proof of the last relation can be found in \cite{qu2017harnessing} Lemma 10.
	
	\subsection{Proof of Lemma \ref{lem: condition u v >0}}
	\label{subsec: proof condition u v >0}
	We first demonstrate that $u_i>0$ iff $i\in\mathcal{R}_R$. Note that there exists an order of vertices such that $R$ can be written as
	\begin{equation}
	\tilde{R}=\begin{Bmatrix}
	R_1 & \mathbf{0}\\
	R_2 & R_3
	\end{Bmatrix}
	\end{equation}
	where $R_1$ is a square matrix corresponding to vertices in $\mathcal{R}_R$. $R_1$ is row stochastic and irreducible (since the associated graph $\mathcal{G}_{R_1}$ is strongly connected). In light of the Perron-Frobenius theorem, $R_1$ has a strictly positive left eigenvector $u_1^{\T}$ ($u_1^{\T} \mathbf{1}=n$) corresponding to eigenvalue $1$. It follows that $[u_1, \mathbf{0}]^{\T}$ is a row eigenvector of $\tilde{R}$, which is also unique from the Perron-Frobenius theorem. Since reordering of vertices does not change the corresponding eigenvector (up to permutation in the same oder of vertices), $u_i>0$ iff $i\in\mathcal{R}_R$.
	
	Similarly, $v_j>0$ iff $j\in\mathcal{R}_{C^{\T}}$. 
	We conclude that $\mathcal{R}_R\cap \mathcal{R}_{C^{\T}}\neq \emptyset$ iff $u^{\T} v>0$.
	
	\subsection{Proof of Lemma \ref{lem: spectral radii}}
	\label{subsec: proof lemma spectral radii}
	In light of \cite[Lemma 3.4]{ren2005consensus}, under Assumptions \ref{asp: stochastic}-\ref{asp: nonempty root set}, spectral radii of $R$ and $C$ are both equal to $1$ (the corresponding eigenvalues have multiplicity $1$). Suppose for some $\lambda, \tilde{u}\neq0$,
	\begin{equation*}
	\tilde{u}^{\T}\left(R-\frac{\mathbf{1}u^{\T}}{n}\right)=\lambda\tilde{u}^{\T}.
	\end{equation*}
	Since $\mathbf{1}$ is a right eigenvector of $(R-\mathbf{1}u^{\T}/n)$ corresponding to eigenvalue $0$, $\tilde{u}^{\T}\mathbf{1}=0$ (see \cite{horn1990matrix} Theorem 1.4.7). We have
	\begin{equation*}
	\tilde{u}^{\T} R=\lambda\tilde{u}.
	\end{equation*}
	Hence $\lambda$ is also an eigenvalue of $R$. Noticing that $u^{\T}\mathbf{1}=n$, we have $\tilde{u}^{\T}\neq u^{\T}$ so that $\lambda<1$. We conclude that $\sigma_R<1$. Similarly we can obtain $\sigma_C<1$.
	
	\subsection{Proof of Lemma \ref{lem: matrix norm production}}
	\label{subsec: proof lemma matrix norm production}
	By Definition \ref{def: norm n p}, 
	\begin{multline}
	\|W\mx\|=\|[\|W\mx^1\|,\|W\mx^2\|,\ldots,\|W\mx^p\|]\|_2\\
	\le \|[\|W\|\|\mx^1\|,\|W\|\|\mx^2\|,\ldots,\|W\|\|\mx^p\|]\|_2\\
	=\|W\|\|[\|\mx^1\|,\|\mx^2\|,\ldots,\|\mx^p\|]\|_2=\|W\|\|\mx\|,
	\end{multline}
	and
	\begin{multline}
	\|wx\|=\|[\|wx^1\|,\|wx^2\|,\ldots,\|wx^p\|]\|_2\\
	=\|w\|\|[|x^1|,|x^2|,\ldots,|x^p|]\|_2=\|w\|\|x\|_2.
	\end{multline}
	
	\addtolength{\textheight}{-12cm}   
	
	\subsection{Proof of Lemma \ref{lem: important inequalities}}
	\label{proof: important inequalities}
	The three inequalities embedded in (\ref{main ineqalities}) come from (\ref{ox pre}), (\ref{mx-ox pre}), and (\ref{my-oy pre}), respectively.
	First, by Lemma \ref{lem: Lipschitz implications} and Lemma \ref{lem: norm equivalence},
	we obtain from (\ref{ox pre}) that
	{\small\begin{multline}
		\|\ox_{k+1}-x^*\|_2\le \|\ox_k-\alpha'g_k-x^*\|_2+\alpha'\|\oy_k-g_k\|_2\\
		+\frac{\alpha}{n}\|(u-\mathbf{1})^{\T}(\my_k-v\oy_k)\|_2\\
		\le (1-\alpha'\mu)\|\ox_k-x^*\|_2+\frac{\alpha'L}{\sqrt{n}}\|\mx_k-\mathbf{1}\ox_k\|_2\\
		+\frac{\alpha\|u-\mathbf{1}\|_2}{n}\|\my_k-v\oy_k\|_2\\
		\le (1-\alpha'\mu)\|\ox_k-x^*\|_2+\frac{\alpha'L}{\sqrt{n}}\|\mx_k-\mathbf{1}\ox_k\|_R\\
		+\frac{\alpha\|u-\mathbf{1}\|_2}{n}\|\my_k-v\oy_k\|_C.
		\end{multline}}\normalsize
	Second, by relation (\ref{mx-ox pre}), Lemma \ref{lem: matrix norm production} and Lemma \ref{lem: norm equivalence}, we see that
	{\small\begin{multline}
		\|\mx_{k+1}-\mathbf{1}\ox_{k+1}\|_R \le \sigma_R\|\mx_k-\mathbf{1}\ox_k\|_R+\alpha\sigma_R\|\my_k-\mathbf{1}\oy_k\|_R\\
		\le \sigma_R\|\mx_k-\mathbf{1}\ox_k\|_R+\alpha\sigma_R\|\my_k-v\oy_k\|_R+\alpha\sigma_R\|v-\mathbf{1}\|_R\|\oy_k\|_2\\
		\le \sigma_R\|\mx_k-\mathbf{1}\ox_k\|_R+\alpha\sigma_R\|\my_k-v\oy_k\|_R\\
		+\alpha\sigma_R\|v-\mathbf{1}\|_R\left(\frac{L}{\sqrt{n}}\|\mx_k-\mathbf{1}\ox_k\|_2+L\|\ox_k-x^*\|_2\right)\\
		\le \sigma_R\left(1+\alpha \|v-\mathbf{1}\|_R \frac{L}{\sqrt{n}}\right)\|\mx_k-\mathbf{1}\ox_k\|_R\\
		+\alpha\sigma_R\delta_{R,C}\|\my_k-v\oy_k\|_C
		+\alpha\sigma_R\|v-\mathbf{1}\|_R L\|\ox_k-x^*\|_2.
		\end{multline}}\normalsize
	Lastly, it follows from (\ref{my-oy pre}),  Lemma \ref{lem: matrix norm production} and Lemma \ref{lem: norm equivalence} that
	{\small\begin{multline}
		\|\my_{k+1}-v\oy_{k+1}\|_C \le \sigma_C\|\my_k-v\oy_k\|_C+\sigma_C \delta_{C,2}L\|\mx_{k+1}-\mx_k\|_2\\
		= \sigma_C\|\my_k-v\oy_k\|_C+\sigma_C \delta_{C,2}L\|(R-I)(\mx_k-\mathbf{1}\ox_k)-\alpha R\my_k\|_2\\
		\le \sigma_C\|\my_k-v\oy_k\|_C+\sigma_C \delta_{C,2}L\|R-I\|_2\|\mx_k-\mathbf{1}\ox_k\|_2\\
		+\alpha\sigma_C \delta_{C,2}L\left\|R(\my_k-v\oy_k)+Rv\oy_k\right\|_2\\
		\le \sigma_C\|\my_k-v\oy_k\|_C+\sigma_C \delta_{C,2}L\|R-I\|_2\|\mx_k-\mathbf{1}\ox_k\|_2\\
		+\alpha\sigma_C \delta_{C,2}L\left(\|R\|_2\|\my_k-v\oy_k\|_2+\|Rv\|_2\|\oy_k\|_2\right).
		\end{multline}}\normalsize
	In light of Lemma \ref{lem: Lipschitz implications},
	{\small
		\begin{equation}
		\|\oy_k\|_2\le \left(\frac{L}{\sqrt{n}}\|\mx_k-\mathbf{1}\ox_k\|_2+L\|\ox_k-x^*\|_2\right).
		\end{equation}}\normalsize
	Hence
	{\small\begin{multline}
		\|\my_{k+1}-v\oy_{k+1}\|_C \le\sigma_C\left(1+\alpha \delta_{C,2}\|R\|_2 L\right)\|\my_k-v\oy_k\|_C\\
		+\sigma_C \delta_{C,2} L\|R-I\|_2\|\mx_k-\mathbf{1}\ox_k\|_2\\
		+\alpha\sigma_C \delta_{C,2}\|Rv\|_2L\left(\frac{L}{\sqrt{n}}\|\mx_k-\mathbf{1}\ox_k\|_2+L\|\ox_k-x^*\|_2\right)\\
		\le \sigma_C\left(1+\alpha \delta_{C,2}\|R\|_2 L\right)\|\my_k-v\oy_k\|_C\\
		+\sigma_C \delta_{C,2} L\left(\|R-I\|_2+\alpha \|Rv\|_2\frac{L}{\sqrt{n}}\right)\|\mx_k-\mathbf{1}\ox_k\|_R\\
		+\alpha\sigma_C \delta_{C,2}\|Rv\|_2 L^2\|\ox_k-x^*\|_2.
		\end{multline}}
	
		\subsection{Proof of Lemma \ref{lem: rho_M}}
	\label{subsec: proof lemma rho_M}
	The characteristic function of $M$ is given by
	\begin{multline}
	g(\lambda):=\text{det}(\lambda I-M)=(\lambda-m_{11})(\lambda-m_{22})(\lambda-m_{33})\\
	-a_{23}a_{32}(\lambda-m_{11})-a_{13}a_{31}(\lambda-m_{22})-a_{12}a_{21}(\lambda-m_{33})\\
	-a_{12}a_{23}a_{31}-a_{13}a_{32}a_{21}.
	\end{multline}
	Necessity is trivial since $\text{det}(\lambda^* I-M)\le 0$ implies $g(\lambda)=0$ for some $\lambda\ge \lambda^*$. We now show $\text{det}(\lambda^* I-M)>0$ is also a sufficient condition.
	
	Given that $g(\lambda^*)=\text{det}(\lambda^* I-M)>0$,
	\begin{multline*}
	(\lambda^*-m_{11})(\lambda^*-m_{22})(\lambda^*-m_{33})\\
	> a_{23}a_{32}(\lambda^*-m_{11})+a_{13}a_{31}(\lambda^*-m_{22})+a_{12}a_{21}(\lambda^*-m_{33}).
	\end{multline*}
	It follows that
	\begin{equation}
	\begin{array}{ccc}
	\gamma_1(\lambda^*-m_{22})(\lambda^*-m_{33}) & > & a_{23}a_{32}\\
	\gamma_2(\lambda^*-m_{11})(\lambda^*-m_{33}) & > & a_{13}a_{31}\\
	\gamma_3(\lambda^*-m_{11})(\lambda^*-m_{22}) & > & a_{12}a_{21}
	\end{array}
	\end{equation}
	for some $\gamma_1,\gamma_2,\gamma_3>0$ with $\gamma_1+\gamma_2+\gamma_3\le 1$.
	Consider
	\begin{multline*}
	g'(\lambda)=(\lambda-m_{22})(\lambda-m_{33})+(\lambda-m_{11})(\lambda-m_{33})\\
	+(\lambda-m_{11})(\lambda-m_{22})-a_{23}a_{32}-a_{13}a_{31}-a_{12}a_{21}.
	\end{multline*}
	We have $g'(\lambda)>0$ for $\lambda\in(-\infty,-\lambda^*]\cup[\lambda^*,+\infty)$. Noticing that
	\begin{multline*}
	g(-\lambda^*)\le -(\lambda^*+m_{11})(\lambda^*+m_{22})(\lambda^*+m_{33})\\
	+a_{23}a_{32}(1+m_{11})+a_{13}a_{31}(\lambda^*+m_{22})\\
	+a_{12}a_{21}(\lambda^*+m_{33})< 0,
	\end{multline*}
	all real roots of $g(\lambda)=0$ lie in the interval $(-\lambda^*,\lambda^*)$. By the Perron-Frobenius theorem, $\rho_M$ is an eigenvalue of $M$. We conclude that $\rho_M<\lambda^*$.
	
\end{document}